\documentclass{article}
\usepackage{amsmath,amsthm,amssymb,amsfonts,bm,mathrsfs}
\usepackage{fancyhdr}
\newtheorem{theorem}{Theorem}[section]
\newtheorem{definition}{Definition}[section]

\newtheorem{lemma}{Lemma}[section]
\newtheorem{question}[theorem]{Question}

\begin{document}
\title{\normalsize\bf Li-Yorke chaos for invertible mappings on  non-compact spaces}

\author{\small Hou Bingzhe\thanks{E-mail:houbz@jlu.edu.cn and Adrress:Department of Mathematics,Jilin University, ChangChun,130012,P. R. China.}\quad and\quad\small Luo Lvlin\thanks{E-mail:luoll12@mails.jlu.edu.cn and Adrress:Department of Mathematics,Jilin University, ChangChun,130012,P. R. China.}}
\date{}
\maketitle
\begin{abstract}
\noindent In this paper, we give two examples to show that an invertible mapping is Li-Yorke chaotic does not imply its inverse being Li-Yorke chaotic, in which
one is an invertible bounded linear operator on an infinite dimensional Hilbert space and the other is a homeomorphism on the unit open disk. Moreover, we use the last
example to prove that Li-Yorke chaos is not preserved under topological conjugacy.



{\noindent\small{\bf Supported by:}} the National Nature Science Foundation of China
(Grant No. 11001099).
\end{abstract}

\section{\normalsize Introduction}
In this paper, we are interested in invertible dynamical system $(X,f)$, where $X$ is a metric space and $f:
X\rightarrow X$ is a homeomorphism. There is a natural problem in invertible dynamical systems as follows.
\begin{question}
Let $(X,f)$ be an invertible dynamical system. If $f$ has a dynamical property $\mathfrak{P}$, does its inverse $f^{-1}$ also have property $\mathfrak{P}$?
\end{question}
It is not difficult to see that the answer is positive for many properties such as
transitivity, mixing and Devaney chaos. However, the conclusion for Li-Yorke chaos, which is defined by Li and Yorke in \cite{LY} in 1975, has not been known.

\begin{definition}
Let $(X,f)$ be a dynamical system. $\{x,y\}\subseteq X$ is said to be a Li-Yorke chaotic pair, if
\begin{equation*}
\limsup\limits_{n\rightarrow+\infty}d(f^{n}(x),f^{n}(y))>0 \ \ and \ \
\liminf\limits_{n\rightarrow+\infty}d(f^{n}(x),f^{n}(y))=0.
\end{equation*}
Furthermore, $f$ is called Li-Yorke chaotic, if there exists an
uncountable subset $\Gamma\subseteq X$ such that each pair of two
distinct points in $\Gamma$ is a Li-Yorke chaotic pair.
\end{definition}

In the present article, we focus on invertible dynamical systems on non-compact metric spaces. Then we study Li-Yorke chaos for invertible mappings on non-compact spaces and
give a negative answer to above question for Li-Yorke chaos. In fact, we give two counterexamples on infinite dimensional space and finite dimensional space respectively.

In section 2, we will give an example to show that an invertible bounded linear operator on infinite dimensional Hilbert spaces is Li-Yorke chaotic does not imply its inverse being Li-Yorke chaotic. In fact,
the example is distributional chaotic and has been introduced in \cite{Hou}.
Distributional chaos that is defined in Schweizer and Sm\'{\i}tal's paper \cite{S-S},
is a kind of chaos stronger than Li-Yorke chaos. Notice that there exists an invertible bilateral forward weighted shift operator $T$ introduced in \cite{LH}, such that $T$ is distributional chaotic but
$T^{-1}$ is not. However, the inverse of the operator $T$ in \cite{LH} is Li-Yorke chaotic.

In section 3, we will give a Li-Yorke chaotic homeomorphism $f$ on the unit open disk such that its inverse $f^{-1}$ is not Li-Yorke chaotic. Moreover, we obtain a homeomorphism $g$ on two dimensional Euclidean space being not Li-Yorke chaotic but topologically conjugate to above $f$, which means that Li-Yorke chaos is not preserved under
topological conjugacy.
\begin{definition}
Let $f:
X\rightarrow X$ and $g: Y\rightarrow Y$ be two continuous mappings.
$f$ is said to be topologically conjugate to $g$, if there exists a
homeomorphism $h: X\rightarrow Y$ such that $h\circ f=g\circ
h$. We also say that $h$ is a topological conjugacy from $f$ to $g$.
\end{definition}
It is
easy to see that many properties are preserved under topological
conjugacy, such as density of periodic points,
transitivity, mixing and Devaney chaos. In \cite{LTX}, Lu et. al. gave a homeomorphism on a discrete topology space such that it is Li-Yorke chaotic
but its inverse is not.

\section{\normalsize Invertible bounded linear operators on infinite dimensional Hilbert spaces}

\begin{theorem}\label{ifnchaos}
There exists an invertible bounded linear operator $T$ on an infinite dimensional Hilbert spaces $\mathcal{H}$,
such that $T$ is Li-Yorke chaotic but $T^{-1}$ is not.
\end{theorem}
\begin{proof}
First of all, let us review the construction of a distributional chaotic operator defined in \cite{Hou}.
Given any $\epsilon>0$. Let $C_i$ be a sequence of positive
numbers increasing to $+\infty$. For each $i\in \mathbb{N}$, set
$\epsilon_i=4^{-i}\epsilon$. Then we can select $L_i$ to satisfy
$(1+\epsilon_i)^{L_i}\geq\sqrt{2}C_i$. Moreover, choose $m_i$ such
that $\frac{L_i}{m_i}<\frac{1}{i}$ and put $n_i=2m_i$.

Let $\mathcal {H}$ be a separable complex Hilbert
space with an orthogonal decomposition $\mathcal {H}=\bigoplus_{i=1}^{\infty}H_i$, where
$H_i$ is an $n_i$-dimensional subspace. For each $i\in \mathbb{N}$, define
$T_i: H_i\rightarrow H_i$ by
$$
T_{i}={\begin{matrix}\begin{bmatrix}
{1-\epsilon_{i}}&2\epsilon_{i}\\
&\ddots&\ddots\\
&&\ddots&2\epsilon_{i}\\
&&&{1-\epsilon_{i}}\\
\end{bmatrix}&
\begin{matrix}
\end{matrix}\end{matrix}}_{(n_{i}\times n_{i})}.
$$
Then we obtain an invertible bounded linear operator
$$
T=\bigoplus_{i=1}^{\infty}T_i:\mathcal {H}\rightarrow\mathcal {H}.
$$
Following from Theorem 9 in \cite{Hou}, $T$ is distributional chaotic(Li-Yorke chaotic). Now consider the inverse of $T$.
Notice that $T^{-1}=\bigoplus_{i=1}^{\infty}T_i^{-1}$ and
$$
T_{i}^{-1}={\begin{matrix}\begin{bmatrix}
{\frac{1}{1-\epsilon_{i}}}&*&\cdots&*\\
&\ddots&\ddots&\vdots\\
&&\ddots&*\\
&&&{\frac{1}{1-\epsilon_{i}}}\\
\end{bmatrix}&
\begin{matrix}
\end{matrix}\end{matrix}}_{(n_{i}\times n_{i})}.
$$
For any $0\neq x=\oplus_{i=1}^{\infty}u_i\in\mathcal {H}$, where $u_i\in H_i$, there exists some certain $u_i\neq0$. Since $T_i^{-1}$ is an upper-triangular matrix with all eigenvalues more than 1,
then
$$
\lim\limits_{n\rightarrow+\infty}\|T^{-n}(x)\|\geq\lim\limits_{n\rightarrow+\infty}\|T_i^{-n}(u_i)\|=+\infty
$$
Therefor, $T^{-1}$ can not be Li-Yorke chaotic.
\end{proof}

\section{\normalsize Invertible nonlinear mappings on finite dimensional spaces}

Denote $\mathbb{R}$ by the set of all real numbers, $\mathbb{Q}$ by the set of all rational numbers.
Notice that $(\mathbb{R},+)$ is an abelian group and $\mathbb{Q}$ is a subgroup of $\mathbb{R}$. Let us consider a simple conclusion of the quotient group $\mathbb{R}/\mathbb{Q}$ firstly.

\begin{lemma}\label{1}
$\mathbb{R}/\mathbb{Q}$ is an uncountable infinite set.
\end{lemma}
\begin{proof}
For each $x\in\mathbb{R}$, denote $[x]$ the equivalence class of $x$ in $\mathbb{R}/\mathbb{Q}$.
Then $[x]=\{x+q; q\in\mathbb{Q}\}$ is a countable finite subset of $\mathbb{R}$. Since $\mathbb{R}$ is uncountable, $\mathbb{R}/\mathbb{Q}$ is an uncountable infinite set.
\end{proof}

\begin{lemma}\label{2}
There exists an uncountable infinite subset $B$ of the open interval $(0,1)$ such that, for any distinct $x,y\in B$, $\ln\frac{x}{1-x}-\ln\frac{y}{1-y}$ is an irrational number.
\end{lemma}
\begin{proof}
Define a homeomorphism $\phi: (0,1)\rightarrow\mathbb{R}$,
$$
\phi(r)=\ln\frac{r}{1-r}, \ \  \ for \ any \ r\in(0,1).
$$
For each element $[x]$ in $\mathbb{R}/\mathbb{Q}$, select one number $x\in[x]$ and consequently denote $A$
the collection of such numbers. Then $A\subseteq\mathbb{R}$ is an uncountable set by Lemma \ref{1}, in which distinct numbers belong to distinct equivalence classes.

Let $B=\{r\in(0,1);\phi(r)\in A\}$.
Then $B$ is uncountable, and for any two distinct number $x,y\in B$,
$$
\ln\frac{x}{1-x}-\ln\frac{y}{1-y}\notin \mathbb{Q}.
$$
\end{proof}

\begin{theorem}\label{fnchaos}
There exist a homeomorphism $f$ on the unit open disk $\mathbb{D}\triangleq\{z\in\mathbb{C};|z|< 1\}$ and a homeomorphism $g$ on $\mathbb{R}^2$,
such that $f$ and $g$ are topologically conjugate, $f$ is Li-Yorke chaotic but $f^{-1}$ is not, $g$ and $g^{-1}$ are not Li-Yorke chaotic.
\end{theorem}
\begin{proof}

Define a mapping $f:\mathbb{D}\rightarrow\mathbb{D}$ by
$$
f(0)=0 \ \ and \ \ f(z)=\frac{ez}{e|z|-|z|+1}e^{2\pi i \ln\frac{|z|}{1-|z|}}, \ \ \ for \ all \ 0\neq z\in \mathbb{D}.
$$
Then $f$ is a homeomorphism on $\mathbb{D}$ and its inverse is
$$
f^{-1}(0)=0 \ \ and \ \ f^{-1}(w)=\frac{e^{-1}w}{e^{-1}|w|-|w|+1}e^{-2\pi i \ln\frac{|w|}{1-|w|}}, \ \ \ for \ all \ 0\neq w\in \mathbb{D}.
$$

Define a mapping $g:\mathbb{R}^2\rightarrow\mathbb{R}^2$ by
$$
g(0)=0 \ \ and \ \ g(z)={ez}e^{2\pi i \ln |z|}, \ \ \ for \ all \ 0\neq z\in \mathbb{R}^2.
$$
Then $g$ is a homeomorphism on $\mathbb{R}^2$ and its inverse is
$$
g^{-1}(0)=0 \ \ and \ \ g^{-1}(w)={e^{-1}w}e^{-2\pi i \ln |w|}, \ \ \ for \ all \ 0\neq w\in \mathbb{R}^2.
$$

Define a mapping $h:\mathbb{D}\rightarrow\mathbb{R}^2$ by
$$
h(z)=\frac{z}{1-|z|}, \ \ \ for \ all \ z\in \mathbb{D}.
$$
Then $h$ is a homeomorphism and its inverse is
$$
h^{-1}(w)=\frac{w}{1+|w|}, \ \ \ for \ all \ w\in \mathbb{R}^2.
$$

One can see
$$
h\circ f(0)=0=g\circ h(0) \ \ and \ \ h\circ f(z)=\frac{ez}{1-|z|}e^{2\pi i \ln\frac{|z|}{1-|z|}}=g\circ h(z),  \ \ for \ all \ 0\neq z\in \mathbb{D}.
$$
Consequently, $h$ is a topological conjugacy from $f$ to $g$.

It is easy to see that for any $k\in \mathbb{Z}$
$$
g^{k}(0)=0 \ \ and \ \ g^{k}(z)={e^{k}z}e^{2\pi i k\ln |z|}, \ \ for \ all \ 0\neq z\in \mathbb{R}^2.
$$
Then,
$$
\lim\limits_{n\rightarrow+\infty}g^{n}(0)=\lim\limits_{n\rightarrow+\infty}g^{-n}(0)=0,
$$
and for any $ 0\neq z\in \mathbb{R}^2$,
$$
\lim\limits_{n\rightarrow+\infty}|g^{n}(z)|=+\infty,
$$
and
$$
\lim\limits_{n\rightarrow+\infty}|g^{-n}(z)|=0.
$$
Therefore, both $g$ and $g^{-1}$ are not Li-Yorke chaotic.

Notice that for any $n\in \mathbb{N}$
$$
f^{-n}(0)=0 \ \ and \ \ f^{-n}(z)=\frac{e^{-n}z}{e^{-n}|z|-|z|+1}e^{2\pi i (-n)\ln\frac{|z|}{1-|z|}}, \ \ \ for \ all \ 0\neq z\in \mathbb{D}.
$$
Then
$$
\lim\limits_{n\rightarrow+\infty}|f^{-n}(z)|=0, \ \ \ for \ all \ 0\neq z\in \mathbb{D}.
$$
Therefore, $f^{-1}$ is not Li-Yorke chaotic.

To complete this proof, it suffices to show $f$ is Li-Yorke chaotic now. Choose an uncountable infinite subset $B$ of the open interval $(0,1)$ defined in Lemma \ref{2}.
Given any two distinct points $x,y\in B$, then $\ln\frac{x}{1-x}-\ln\frac{y}{1-y}$ is an irrational number. Consequently, there exist two sequences of strictly increasing positive integers
$\{m_k\}$ and $\{n_k\}$ such that
$$
\lim\limits_{k\rightarrow+\infty}e^{2\pi i m_k(\ln\frac{x}{1-x}-\ln\frac{y}{1-y})}=1 \ \ and \ \
\lim\limits_{k\rightarrow+\infty}e^{2\pi i n_k(\ln\frac{x}{1-x}-\ln\frac{y}{1-y})}=-1.
$$
Since for any $n\in \mathbb{N}$,
$$
f^{n}(z)=\frac{e^{n}z}{e^{n}|z|-|z|+1}e^{2\pi i n\ln\frac{|z|}{1-|z|}} \ \ \ for \ all \ 0\neq z\in \mathbb{D},
$$
and
$$
\lim\limits_{n\rightarrow+\infty}\frac{e^{n}|z|}{e^{n}|z|-|z|+1}=1 \  \ for \ all \ 0\neq z\in \mathbb{D},
$$
then
\begin{eqnarray*}
&& \liminf\limits_{n\rightarrow +\infty}|f^{n}(x)-f^{n}(y)|  \\
&\leq& \lim\limits_{k\rightarrow +\infty}|f^{m_k}(x)-f^{m_k}(y)|  \\
&=& \lim\limits_{k\rightarrow+\infty}|\frac{e^{m_k}x}{e^{m_k}x-x+1}e^{2\pi i {m_k}\ln\frac{x}{1-x}}-\frac{e^{m_k}y}{e^{m_k}y-y+1}e^{2\pi i {m_k}\ln\frac{y}{1-y}}|  \\
&\leq& \lim\limits_{k\rightarrow+\infty}|\frac{e^{m_k}x}{e^{m_k}x-x+1}e^{2\pi i {m_k}\ln\frac{x}{1-x}}-\frac{e^{m_k}y}{e^{m_k}y-y+1}e^{2\pi i {m_k}\ln\frac{x}{1-x}}|+  \\
&& \lim\limits_{k\rightarrow +\infty}|\frac{e^{m_k}y}{e^{m_k}y-y+1}e^{2\pi i {m_k}\ln\frac{x}{1-x}}-\frac{e^{m_k}y}{e^{m_k}y-y+1}e^{2\pi i {m_k}\ln\frac{y}{1-y}}|  \\
&=& \lim\limits_{k\rightarrow+\infty}|\frac{e^{m_k}x}{e^{m_k}x-x+1}-\frac{e^{m_k}y}{e^{m_k}y-y+1}||e^{2\pi i {m_k}\ln\frac{x}{1-x}}|+   \\
&&\lim\limits_{k\rightarrow +\infty}|\frac{e^{m_k}y}{e^{m_k}y-y+1}||e^{2\pi i {m_k}\ln\frac{y}{1-y}}||e^{2\pi i {m_k}(\ln\frac{x}{1-x}-\ln\frac{y}{1-y})}-1|   \\
&=& 0.
\end{eqnarray*}
and
\begin{eqnarray*}
&& \limsup\limits_{n\rightarrow +\infty}|f^{n}(x)-f^{n}(y)|  \\
&\geq& \lim\limits_{k\rightarrow +\infty}|f^{n_k}(x)-f^{n_k}(y)|  \\
&=& \lim\limits_{k\rightarrow+\infty}|\frac{e^{n_k}x}{e^{n_k}x-x+1}e^{2\pi i {n_k}\ln\frac{x}{1-x}}-\frac{e^{n_k}y}{e^{n_k}y-y+1}e^{2\pi i {n_k}\ln\frac{y}{1-y}}|  \\
&\geq& \lim\limits_{k\rightarrow+\infty}|\frac{e^{n_k}y}{e^{n_k}y-y+1}e^{2\pi i {n_k}\ln\frac{x}{1-x}}-\frac{e^{n_k}y}{e^{n_k}y-y+1}e^{2\pi i {n_k}\ln\frac{y}{1-y}}|-  \\
&& \lim\limits_{k\rightarrow +\infty}|\frac{e^{n_k}x}{e^{n_k}x-x+1}e^{2\pi i {n_k}\ln\frac{x}{1-x}}-\frac{e^{n_k}y}{e^{m_k}y-y+1}e^{2\pi i {n_k}\ln\frac{x}{1-x}}|  \\
&=& \lim\limits_{k\rightarrow +\infty}|\frac{e^{n_k}y}{e^{n_k}y-y+1}||e^{2\pi i {n_k}\ln\frac{y}{1-y}}||e^{2\pi i {n_k}(\ln\frac{x}{1-x}-\ln\frac{y}{1-y})}-1|-   \\
&&\lim\limits_{k\rightarrow+\infty}|\frac{e^{n_k}x}{e^{n_k}x-x+1}-\frac{e^{n_k}y}{e^{n_k}y-y+1}||e^{2\pi i {n_k}\ln\frac{x}{1-x}}|  \\
&=& 2.
\end{eqnarray*}
Thus, $\{x,y\}$ is a Li-Yorke chaotic pair and hence $f$ is Li-Yorke chaotic.
\end{proof}

\end{document}